\documentclass[12pt]{article}
\usepackage{amsmath}
\usepackage{amsthm}
\usepackage{amsfonts}
\usepackage{amssymb}
\usepackage{esint}
\usepackage{eucal}
\usepackage{mathrsfs}
\usepackage{bm}
\usepackage{graphicx,graphics}
\numberwithin{equation}{section}
\usepackage{graphicx,graphics}
\usepackage{color}
\usepackage{pdfsync}
\usepackage{multirow}
\usepackage{endnotes}
\usepackage{epstopdf}
\def \dis {\displaystyle}

\def \confai {-\kern -.5em\rightharpoonup}
\def \cqfd {\hfill$\Box$}
\def \div{\mbox{\rm div}\,}
\def \curl{\mbox{\rm curl}\,}

\def \al {\alpha}

\def \de {\delta}
\def \De {\Delta}

\def \om {\omega}
\def \Om {\Omega}

\def \ph {\varphi}

\def \si {\sigma}

\def \NN {\mathbb N}
\def \ZZ {\mathbb Z}

\def \RR {\mathbb R}

\def \D {\mathscr{D}}

\def \beq {\begin{equation}}
\def \eeq {\end{equation}}
\def \ba {\begin{array}}
\def \ea {\end{array}}
\def \bs {\bigskip}
\def \ms {\medskip}
\def \ss {\smallskip}
\def \ecart {\noalign{\medskip}}
\newtheorem{Thm}{Theorem}[section]
\newtheorem{Cor}[Thm]{Corollary}
\newtheorem{Pro}[Thm]{Proposition}

\newtheorem{Adef}[Thm]{Definition}
\newenvironment{Def}{\begin{Adef}\rm}{\end{Adef}}
\newtheorem{Arem}[Thm]{Remark}
\newenvironment{Rem}{\begin{Arem}\rm}{\end{Arem}}
\newtheorem{Aappl}[Thm]{Application}

\newtheorem{Aexa}[Thm]{Example}
\newenvironment{Exa}{\begin{Aexa}\rm}{\end{Aexa}}
\newtheorem{Anot}[Thm]{Notation}

\def \refe #1.{(\ref{#1})}
\def \reff #1.{figure~\ref{#1}}
\def \refs #1.{Section~\ref{#1}}
\def \refss #1.{Subsection~\ref{#1}}
\def \refD #1.{Definition~\ref{#1}}
\def \refT #1.{Theorem~\ref{#1}}
\def \refL #1.{Lemma~\ref{#1}}
\def \refC #1.{Corollary~\ref{#1}}
\def \refP #1.{Proposition~\ref{#1}}
\def \refPt #1.{Properties~\ref{#1}}
\def \refR #1.{Remark~\ref{#1}}
\def \refA #1.{Application~\ref{#1}}
\def \refE #1.{Example~\ref{#1}}
\def \refN #1.{Notation~\ref{#1}}
\newcounter{marnote}

\title{Isotropic realizability of current fields in $\RR^3$}
\author{M. Briane\footnote{INSA de Rennes, IRMAR (CNRS, UMR 6625), FRANCE -- mbriane@insa-rennes.fr}
\and G.W. Milton\footnote{Department of Mathematics, University of Utah, USA -- milton@math.utah.edu}
}
\begin{document}
\maketitle
%\tableofcontents
\begin{abstract}
This paper deals with the isotropic realizability of a given regular divergence free field $j$ in $\RR^3$ as a current field, namely to know when $j$ can be written as $\si\nabla u$ for some isotropic conductivity $\si>0$, and some gradient field $\nabla u$. The local isotropic realizability in $\RR^3$ is obtained by Frobenius' theorem provided that $j$ and $\curl j$ are orthogonal in~$\RR^3$. A counter-example shows that Frobenius' condition is not sufficient to derive the global isotropic realizability in $\RR^3$. However, assuming that $(j,\curl j,j\times\curl j)$ is an orthogonal basis of $\RR^3$,  an admissible conductivity $\si$ is constructed from a combination of the three dynamical flows along the directions $j/|j|$, $\curl j/|\curl j|$ and $(j/|j|^2)\times\curl j$. When the field $j$ is periodic, the isotropic realizability in the torus needs in addition a boundedness assumption satisfied by the flow along the third direction $(j/|j|^2)\times\curl j$. Several examples illustrate the sharpness of the realizability conditions.
\end{abstract}
\par\bs\noindent
\section{Introduction}
In the theory of composite conductors (see, {\em e.g.}, \cite{Mil}), we are naturally led to study periodic composites. The effective properties of a periodic composite are obtained by passing from a local Ohm's law
\beq\label{Ohm}
j=\si\,e,
\eeq
between a periodic divergence free current field $j$ and a periodic electric field $e$, to an effective Ohm's law
\beq
\langle j\rangle=\si^*\langle e\rangle,\quad\mbox{with $\langle\cdot\rangle$ the average over the period cell},
\eeq
where $\si(y)$ is the local conductivity which is isotropic, and $\si^*$ is the (constant) effective conductivity of the composite which is in general anisotropic. In this context, it is natural to characterize the periodic current fields arising in the solution of these equations among all the divergence free fields. More precisely, the paper deals with the following question: Given a periodic regular divergence free field $j$ from $\RR^3$ into $\RR^3$, under which conditions $j$ is an isotropically realizable current field, namely there exists an isotropic conductivity $\si>0$ and a gradient field such that $j=\si\nabla u$?
An additional motivation comes from the success of transformation optics (see, {\em e.g.}, \cite{PSS1,PSS2}) where the objective is to choose moduli (in our case the conductivity~$\sigma$) to achieve desired fields (in our case the prescribed current field $j$).
\par
In \cite{BMT} we have studied the isotropic realizability of a given regular electric field $e=\nabla u$ in~$\RR^d$, for any $d\geq 2$. The key ingredient of our approach was the associated gradient system 
\beq\label{X}
\left\{\ba{ll}
X'(t,x)=\nabla u\big(X(t,x)\big), & \mbox{for }t\in\RR,
\\ \ecart
X(0,x)=x, &
\ea\right.
\eeq
which allowed us to prove the following isotropic realizability result for a gradient field in the whole space and in the torus:
\begin{Thm}[\cite{BMT}, Theorems~2.15 \& 2.17]\label{thm.elerea}
Let $u$ be a function in $C^3(\RR^d)$ satisfying the non-vanishing condition
\beq\label{Du�0}
\inf_{\RR^d}|\nabla u|>0.
\eeq
Then, there exists a unique function $\tau\in C^1(\RR^d)$ such that for any $x\in\RR^d$, the trajectory $t\mapsto X(t,x)$ meets the equipotential $\{u=0\}$ at the times $\tau(x)$, namely
\beq\label{tau(x)}
u\big(X(\tau(x),x)\big)=0.
\eeq
Moreover, the positive function $\si$ defined in $\RR^d$ by
\beq\label{sitau}
\si(x):=\exp\left(\int_0^{\tau(x)}\De u\big(X(s,x)\big)\,ds\right),\quad\mbox{for }x\in\RR^d,
\eeq
satisfies the conductivity equation $\div(\si\nabla u)=0$ in $\RR^d$.
\par
On the other hand, when $\nabla u$ is periodic, the conductivity $\si$ can be chosen periodic if and only if there exists a constant $C>0$ such that
\beq\label{bdDeu}
\left|\,\int_0^{\tau(x)}\De u\big(X(t,x)\big)\,dt\,\right|\leq C,\quad \forall\,x\in\RR^d.
\eeq
\end{Thm}
\noindent
In the case of a gradient field, the local isotropic realizability which follows from the non-vanishing condition \eqref{Du�0} thanks to the rectification theorem (see \cite{BMT}, Theorem~2.2 $i)$), is thus equivalent to the global realizability given by the previous theorem.
\par
The case of a regular divergence free field $j$ in $\RR^3$ is much more intricate. First of all, a necessary condition for the isotropic realizability is the orthogonality of $j$ and $\curl j$ in~$\RR^3$. Conversely, if $j$ is non-zero and orthogonal to $\curl j$ in $\RR^3$, then Frobenius' theorem implies that $j$ is isotropically realizable locally in $\RR^3$ (see Proposition~\ref{pro.locrea}). However, contrary to the case of a gradient field, these two conditions  are not sufficient to ensure the global realizability (see Section~\ref{ss.cexF} for a counter-example). This strictly local nature of Frobenius' theorem is strongly connected to cohomology which is outside the scope of this paper. On the other hand, we cannot use for a current field the properties of a gradient system which permits us in particular to define the time $\tau(x)$ satisfying \eqref{tau(x)}.
\par
Our approach concerning the isotropic realizability of a current field is still based on dynamical systems. But now, the procedure to construct an admissible conductivity associated with a given regular divergence free field $j$, uses a combination of three dynamical flows which are not of gradient type. To this end, we need that the three fields $j$, $\curl j$ and $j\times\curl j$ make an orthogonal basis of $\RR^3$, including in this way Frobenius' condition. Then, the method consists in flowing from a fixed point $x_0\in\RR^3$, first with the flow $X_1$ along the direction $j/|j|$ during a time $t_1$, next with the flow $X_2$ along the direction $\curl j/|\curl j|$ during a time $t_2$, finally with the flow $X_3$ along the direction $j/|j|^2\times\curl j$ during a time $t_3$. So, we obtain the triple time dynamical flow
\beq\label{Xti}
(t_1,t_2,t_3)\mapsto X_{32}(t_3,t_2,t_1)=X_3\big(t_3,X_2(t_2,X_1(t_1))\big),\quad\mbox{with}\quad X_{32}(0,0,0)=x_0,
\eeq
which is assumed to be a $C^1$-diffeomorphism onto $\RR^3$. Under these assumptions we prove (see Theorem~\ref{thm.glorea}) that the field $j$ is isotropically realizable with the conductivity $\si$ defined by
\beq\label{siX}
\si\big(X_{32}(t_1,t_2,t_3)\big):=\exp\left(\int_0^{t_3}{|\curl j|^2\over|j|^2}\big(X_{32}(s,t_2,t_1)\big)\,ds\right),\quad\mbox{for }(t_1,t_2,t_3)\in\RR^3,
\eeq
This result can be regarded as a global Frobenius' theorem, and is illustrated by the very simple current field $j$ of Section~\ref{ss.jsh}, which yields an infinite set of (not obvious) admissible conductivities. Unfortunately, Section~\ref{ss.jplan} shows that the approach with the triple time dynamical flow fails for a periodic regular field $j$ of a particular form which is everywhere perpendicular to a constant vector, since $\curl j$ does vanish in $\RR^3$ (see Remark~\ref{rem.curlj=0}). However, in this case the divergence free field can be written as an orthogonal gradient, which allows us to apply Theorem~\ref{thm.elerea} for a two-dimensional electric field.
\par
When the field $j$ is periodic, the isotropic realizability in the torus needs an extra assumption as in the case of a periodic electric field \cite{BMT} (Theorem~2.17). Under the former conditions which ensure the isotropic realizability of $j$ in the whole space $\RR^3$, we prove (see Corollary~\ref{cor.perrea}) that the field $j$ is isotropically realizable in the torus, namely reads as $\si\nabla u$ with both $\si$ and $\nabla u$ periodic, if and only if
\beq\label{bdX}
\sup_{(t_1,t_2)\in\RR^2}\left(\int_{-\infty}^{\infty}{|\curl j|^2\over |j|^2}\big(X_{32}(s,t_2,t_1)\big)\,ds\right)<\infty,
\eeq
which is equivalent to the boundedness from below and above of the conductivity~\eqref{siX} in $\RR^3$.
The sharpness of condition \eqref{bdX} is illustrated by Proposition~\ref{pro.fgh} and Example~\ref{rem.f} below.
\par
The paper is divided in two parts. In Section~\ref{s.rea} we study the validity of the isotropic realizability of a regular divergence free field first in the whole space $\RR^3$, then in the torus when the current field is assumed to be periodic. Section~\ref{s.exa} is devoted to examples and counter-examples which illustrate the theoretical results of Section~\ref{s.rea}.
\subsubsection*{Notations}
\begin{itemize}
\item $Y:=[0,1]^3$ and $Y':=[0,1]^2$.
\item $\langle\cdot\rangle$ denotes the average over $Y$.
\item $C^k_\sharp(Y)$ denotes the space of $k$-continuously differentiable $Y$-periodic functions on $\RR^d$.
\item $L^2_\sharp(Y)$ denotes the space of $Y$-periodic functions in $L^2_{\rm }(\RR^d)$, and $H^1_\sharp(Y)$ denotes the space of functions $\ph\in L^2_\sharp(Y)$ such that $\nabla\ph\in L^2_\sharp(Y)^d$.
\item For any open set $\Om$ of $\RR^d$, $C^\infty_c(\Om)$ denotes the space of  smooth functions with compact support in $\Om$, and $\D'(\Om)$ the space of distributions on $\Om$.
\end{itemize}
\section{Results of isotropic realizability}\label{s.rea}
\subsection{Realizability in the whole space}
Let us start by the following definition:
\begin{Def}
Let $j$ be a divergence free field in $L^\infty(\RR^3)^3$ -- $j$ will be taken regular in the sequel --
The field $j$ is said to be {\em isotropically realizable in} $\RR^3$ as a current field, if there exist an isotropic conductivity $\si>0$ with $\si,\si^{-1}\in L^\infty(\RR^3)$, and a potential $u\in W^{1,\infty}(\RR^3)^3$, such that $j=\si\nabla u$. Moreover, when $j$ is $Y$-periodic, $j$ is said to be {\em isotropically realizable in the torus} if $\si$ and $\nabla u$ can be chosen $Y$-periodic.
\end{Def}
\par
First of all we have the following result which provides a criterion for the local isotropic realizability of a regular current field:
\begin{Pro}\label{pro.locrea}
Let $j$ be a vector-valued function such that
\beq\label{divj=0}
j\in C^2(\RR^3)^3,\quad j\neq 0\;\;\mbox{in }\RR^3,\quad\mbox{and}\quad\div j=0\;\;\mbox{in }\RR^3.
\eeq
Then, a necessary and sufficient condition for the current field $j$ to be locally isotropically realizable in $\RR^3$ with some positive $C^1$ conductivity $\si$ is that
\beq\label{j.curlj=0}
j\cdot\curl j=0\quad\mbox{in }\RR^3.
\eeq
\end{Pro}
\begin{proof}
If $j$ is isotropically realizable with some conductivity $\si\in C^1(\RR^3)$, then $j=\si\nabla u$ with $\si\in C^1(\RR^3)$, and
\beq\label{curlj.j=0}
\curl j=\curl(\si\nabla u)=\nabla\si\times\nabla u+\si\,\curl(\nabla u)={\nabla\si\over\si}\times j\quad\mbox{ in }\RR^3,
\eeq
which yields immediately \eqref{j.curlj=0}. Conversely, if \eqref{divj=0} and \eqref{j.curlj=0} are both satisfied, then by Frobenius' theorem (see, {\em e.g.}, \cite{Car}, Theorem~6.6.2 and example p.~279) there exists locally a non-zero $C^1$ function $\si$ and a $C^1$ function $u$, such that $j=\si\nabla u$. The function $\si$ can be chosen positive by a continuity argument, which shows the isotropic realizability of $j$ locally in $\RR^3$.
Actually, the divergence free condition is not necessary to obtain the local realizability.
\end{proof}
Frobenius' condition \eqref{j.curlj=0} implies the local isotropic realizability for a current field $j$ satisfying condition \eqref{divj=0}. However, contrary to the case \cite{BMT} of an electric field for which the local realizability and the global realizability turn out to be equivalent, these two conditions are not sufficient to ensure the global isotropic realizability of $j$, as shown by the counter-example of Section~\ref{ss.cexF}. To overcome this difficulty we will use an alternative approach based on the flows along the three orthogonal directions $j$, $\curl j$ and $j\times\curl j$ under suitable assumptions which are detailed below:
\par
Let $j$ be a current field satisfying conditions \eqref{divj=0}. Beyond condition \eqref{j.curlj=0} we assume that
\beq\label{ortbasj}
\big(j,\curl j,j\times\curl j\big)\;\;\mbox{is an orthogonal basis of }\RR^3.
\eeq
Then, for a fixed $x_0\in\RR^3$ and for any $(t_1,t_2,t_3)\in\RR^3$, consider the flows $X_1(t,x)$, $X_2(t,x)$, $X_3(t,x)$ along the orthogonal directions $j/|j|$, $\curl j/|\curl j|$, $j/|j|^2\times\curl j$ respectively, that is
\beq\label{Xi}
\left\{\ba{ll}
\dis {\partial X_1\over\partial t}(t,x)={j\over|j|}\big(X_1(t,x)\big), & X_1(0)=x,
\\ \ecart
\dis {\partial X_2\over\partial t}(t,x)={\curl j\over|\curl j|}\big(X_2(t,x)\big), & X_2(0)=x,
\\ \ecart
\dis {\partial  X_3\over\partial t}(t,x)={j\times \curl j\over|j|^2}\big(X_3(t,x)\big), & X_3(0,x)=x.
\ea\right.
\quad\mbox{for }(t,x)\in\RR\times\RR^3.
\eeq
Note that the flows $X_1$ and $X_2$ are well defined in the whole set $\RR\times\RR^3$, since by \eqref{divj=0} ${j/|j|}$ and ${\curl j/|\curl j|}$ belong to $C^1(\RR^3)$ and are bounded in $\RR^3$ (see, {\em e.g.}, \cite{Arn}, Chap.~2.6). In the sequel, we will assume that the flow $X_3$ is also defined in the whole set $\RR\times\RR^3$. That is the case if for example $j\in C^2_\sharp(Y)^3$.
\begin{Rem}\label{rem.X3}
In view of the normalization of the flows $X_1$, $X_2$, and to avoid the latter assumption, it seems {\em a~priori} more logical to renormalize the flow $X_3$ with $|j||\curl j|$ rather than $|j|^2$. The derivation of the isotropic realizability is quite similar in both cases (see Theorem~\ref{thm.glorea} and Remark~\ref{rem.tX3} just below). However, the normalization by $|j|^2$ arises naturally in the orthogonal decomposition \eqref{jxcurljDw} which is a key-ingredient for the construction of an admissible conductivity associated with the isotropic realizability of $j$. Moreover, it gives a necessary condition for the isotropic realizability in the torus without the need to assume that $\curl j$ does not vanish in~$\RR^3$ (see the first part of Corollary~\ref{cor.perrea} below). Actually, there are lots of examples where $\curl j$ vanishes somewhere (see Section~\ref{ss.jplan} below), but the normalization of the flow $X_3$ by $|j|^2$ may be relevant in some cases (see Proposition~\ref{pro.jz=0} and Remark~\ref{rem.curlj=0}).
\end{Rem}
Next, denote for a fixed point $x_0\in\RR^3$,
\beq\label{X1X23}
\left\{\ba{ll}
X_1(t_1):=X_1(t_1,x_0)
\\ \ecart
X_{23}(s_2,s_3,t_1):=X_2\big(s_2,X_3(s_3,X_1(t_1,x_0))\big)
\\ \ecart
X_{32}(t_3,t_2,t_1):=X_3\big(t_3,X_2(t_2,X_1(t_1,x_0))\big),
\ea\right.
\quad\mbox{for }(s_1,s_2,s_3,t_1,t_2,t_3)\in\RR^6.
\eeq
So, the dynamical flow $X_{32}$ is obtained by flowing from the point $x_0$ along the direction $j/|j|$ during the time $t_1$, then from the point $X_1(t_1)$ along the direction $\curl j/|\curl j|$ during the time $t_2$, finally from the point $X_2(t_2,X_1(t_1))$ along the direction $(j/|j|^2)\times\curl j$ during the time $t_3$. The end point is thus $X_{32}(t_3,t_2,t_1)$. A similar construction holds for $X_{23}(t_2,t_3,t_1)$ by commuting the flows $X_2$ and $X_3$.
Now, the main assumption is that any point $x$ in $\RR^3$ can be attained by the composition of the three flows, so that $x$ can be represented in a unique way by the system of coordinates $(t_1,t_2,t_3)$, that is
\beq\label{diffX32}
(t_1,t_2,t_3)\mapsto X_{32}(t_3,t_2,t_1)\quad\mbox{is a $C^1$-diffeomorphism onto $\RR^3$.}
\eeq
\par
Then, we have the following sufficient condition for the global isotropic realizability in $\RR^3$:
\begin{Thm}\label{thm.glorea}
Let $j$ be a field in $\RR^3$ satisfying \eqref{divj=0} and \eqref{ortbasj}.
Also assume that condition \eqref{diffX32} holds true.
Then, the field $j$ is isotropically realizable in $\RR^3$ with the conductivity $e^w\in C^1(\RR^3)$, where the function $w$ is defined by
\beq\label{w(x)}
w(x)=w\big(X_{32}(t_3,t_2,t_1)\big):=\int_0^{t_3}{|\curl j|^2\over|j|^2}\big(X_{32}(s,t_2,t_1)\big)\,ds,\quad\mbox{for }(t_1,t_2,t_3)\in\RR^3.
\eeq
\end{Thm}
\begin{Rem}\label{rem.tX3}
In view of Remark~\ref{rem.X3}, if we renormalize the flow $X_3$ by $|j||\curl j|$, it is replaced by the flow $\tilde{X}_3$ defined by
\beq\label{tX3}
{\partial\tilde{X}_3\over\partial t}(t,x)={j\times \curl j\over|j||\curl j|}\big(\tilde{X}_3(t,x)\big),\quad \tilde{X}_3(0,x)=x,
\eeq
which by condition \eqref{ortbasj} is defined in $\RR\times\RR^3$. Then, similarly to \eqref{diffX32}, assuming  that the triple flow
\beq\label{tX32}
\tilde{X}_{32}:(t_1,t_2,t_3)\mapsto \tilde{X}_3\big(t_3,X_2(t_2,X_1(t_1,x_0)\big)
\eeq
is a $C^1$-diffeomorphism onto $\RR^3$, we obtain that  the field $j$ is isotropically realizable in $\RR^3$ with the conductivity $e^{\tilde{w}}\in C^1(\RR^3)$, where
\beq\label{tw(x)}
\tilde{w}(x)=w\big(\tilde{X}_{32}(t_3,t_2,t_1)\big):=\int_0^{t_3}{|\curl j|\over|j|}\big(\tilde{X}_{32}(s,t_2,t_1)\big)\,ds,\quad\mbox{for }(t_1,t_2,t_3)\in\RR^3.
\eeq
The proof is quite similar to the proof of Theorem~\ref{thm.glorea} replacing formula \eqref{Dtwjxcurlj} by
\beq\label{Dtwjxcurlj}
\nabla\tilde{w}\cdot\left({j\times\curl j\over|j||\curl j|}\right)={|\curl j|\over|j|}
=\left({j\times\curl j\over|j|^2}\right)\cdot\left({j\times\curl j\over|j||\curl j|}\right)\quad\mbox{in }\RR^3.
\eeq
\end{Rem}
\begin{Rem}\label{rem.homX32}
Alternatively, we can replace the diffeomorphism condition \eqref{diffX32} by
\beq\label{homX32}
(t_1,t_2,t_3)\mapsto X_{32}(t_3,t_2,t_1)\quad\mbox{is a homeomorphism onto $\RR^3$ of class $C^1$,}
\eeq
so that the Jacobian of $X_{32}$ may vanish somewhere. In compensation we have to assume that the function $w$ of \eqref{w(x)} belongs to $C^1(\RR^3)$. See the application to the example of Section~\ref{ss.jsh}.
\end{Rem}
\begin{Rem}
Condition \eqref{diffX32} is not sharp to ensure the isotropic realizability of the current field.
Indeed, the planar example of Proposition~\ref{pro.jz=0} below shows that the isotropic realizability can be satisfied while condition \eqref{ortbasj} is violated. See also Example~\ref{rem.f} below.
\end{Rem}
\par\noindent
{\bf Proof of Theorem \ref{thm.glorea}.}
\par\ss\noindent
{\it First step: } Construction of an admissible conductivity.
\par\ss\noindent
Let $j$ be a field satisfying \eqref{divj=0}.
Assume that $j$ is isotropically realizable in $\RR^3$, namely there exists $u,w\in C^1(\RR^3)$ such that
\beq\label{juw}
j=e^w\,\nabla u\quad\mbox{in }\RR^3.
\eeq
It seems that we can choose the potential $u$ arbitrarily along the trajectory $X_1(t)$, provided $\nabla u$ does not vanish along this trajectory.
So, define
\beq\label{ut1}
u\big(X_1(t)\big):=\int_0^{t}\big|j\big(X_1(s)\big)\big|\,ds,\quad\mbox{for }t\in\RR.
\eeq
Taking the derivative with respect to $t$ of \eqref{ut1} and using \eqref{Xi}, \eqref{juw}, we get that
\beq
{\partial\over\partial t}\left[u\big(X_1(t)\big)\right]=\big|j\big(X_1(t)\big)\big|=\nabla u\big(X_1(t)\big)\cdot{\partial\over\partial t}\,X_1(t)
=e^{-w(X_1(t))}\,{j\cdot j\over|j|}\big(X_1(t)\big),
\eeq
which implies that
\beq\label{wXt1}
w\big(X_1(t)\big)=0,\quad\mbox{for }t\in\RR.
\eeq
Next, taking the curl of \eqref{juw} we get that $\curl j=\nabla w\times j$, hence
\beq\label{jxcurljDw}
{j\times\curl j\over|j|^2}=\nabla w-\Pi_j(\nabla w)\quad\mbox{in }\RR^3,
\eeq
where $\Pi_j$ is the orthogonal projection on the subspace $\RR j$.
Hence, integrating \eqref{jxcurljDw} along the trajectory $X_3$, we have
\[
\ba{l}
\dis w\big(X_{32}(t_3,t_2,t_1)\big)-w\big(X_{32}(0,t_2,t_1)\big)=w\big(X_{32}(t_3,t_2,t_1)\big)-w\big(X_2(t_2,X_1(t_1))\big)
\\ \ecart
\dis =\int_0^{t_3}\nabla w\big(X_{32}(s,t_2,t_1)\big)\cdot{\partial X_{32}\over\partial s}(s,t_2,t_1)\,ds
=\int_0^{t_3}\left|{j\times\curl j\over|j|^2}\right|^2\big(X_{32}(s,t_2,t_1)\big)\,ds
\\ \ecart
\dis =\int_0^{t_3}{|\curl j|^2\over|j|^2}\big(X_{32}(s,t_2,t_1)\big)\,ds.
\ea
\]
Then, integrating \eqref{jxcurljDw} along the trajectory $X_2$ and using \eqref{j.curlj=0}, we get that
\[
\ba{l}
\dis w\big(X_2(t_2,X_1(t_1))\big)-w\big(X_2(0,X_1(t_1))\big)=w\big(X_2(t_2,X_1(t_1))\big)-w\big(X_1(t_1)\big)
\\ \ecart
\dis =\int_0^{t_2}\nabla w\big(X_2(s,t_1)\big)\cdot{\partial X_2\over\partial s}(s,t_1)\,ds
=\int_0^{t_2}\nabla w\big(X_2(s,t_1)\big)\cdot{\curl j\over|\curl j|}\big(X_2(s,t_1)\big)\,ds=0.
\ea
\]
The two previous equalities combined with \eqref{wXt1} yield the desired expression \eqref{w(x)}.
%which well defines a continuous function due to condition \eqref{diffX32}.
\par\ms\noindent
{\it Second step: } Construction of a grid on the surface $\{t_1=c_1\}$, generated by the flows $X_2,X_3$.
\par\ss\noindent
Let us prove that for any $(c_1,t_2',t_2,t_3)\in\RR^4$, the flows $X_2$ and $X_3$ generate on the regular surface $\{t_1\!=\!c_1\}$, a thin grid whose:
\begin{itemize}
\item step is of small enough size $\nu>0$,
\item horizontal lines are trajectories along the flow $X_2$,
\item vertical lines are trajectories along the flow $X_3$,
\item two opposite vertices are the points
\beq\label{x'x}
x':=X_2(t_2',X_1(c_1))\quad\mbox{and}\quad x:=X_3(t_3,X_2(t_2,x'))=X_{32}(t_3,t_2'+t_2,c_1).
\eeq
\end{itemize}
\noindent
First, we divide the flows in small time steps, as shown the diagram
\beq\label{x'hvx}
\ba{cc}
& \kern -.2cm x
\\
& \kern -.2cm \big\uparrow
\\
& \kern -.2cm \bullet
\\
& \kern -.2cm \big\uparrow
\\
& \kern -.2cm \bullet
\\
& \kern -.2cm \big\uparrow
\\
& \kern -.2cm \bullet
\\
& \kern -.2cm \big\uparrow
\\
x' \longrightarrow \bullet \longrightarrow \bullet \longrightarrow \bullet \longrightarrow \bullet \longrightarrow & \kern -.2cm \bullet
\ea
\eeq
where the horizontal arrows represent the flow $X_2$, and the vertical ones the flow $X_3$.
Then, we commute step by step the flows $X_2$ and $X_3$, while remaining on the surface $\{t_1\!=\!c_1\}$, as shown the commutation diagram
\beq\label{compq}
\ba{cccccccccc}
& \cdots & q & & & & \bullet & \stackrel{^{\textstyle X_2}}{\longrightarrow} & q &
\\
\vdots & & \big\uparrow & \kern -.3cm X_3 & \quad\Longrightarrow\quad & X_3 \kern -.3cm & \big\uparrow & & \big\uparrow
& \kern -.3cm X_3
\\
p & \dis \mathop{\longrightarrow}_{\textstyle X_2} & \bullet & & & & p & \dis \mathop{\longrightarrow}_{\textstyle X_2} & \bullet &
\ea
\eeq
to finally obtain the desired grid
\beq\label{Gx'xy'y}
\ba{ccccccccccc}
y' & \kern -.2cm \longrightarrow & \kern -.2cm \bullet & \kern -.2cm \longrightarrow & \kern -.2cm \bullet & \kern -.2cm \longrightarrow
& \kern -.2cm \bullet & \kern -.2cm \longrightarrow & \kern -.2cm \bullet & \kern -.2cm \longrightarrow & \kern -.2cm x
\\
\big\uparrow & & \kern -.2cm \big\uparrow & & \kern -.2cm \big\uparrow & & \kern -.2cm \big\uparrow & & \kern -.2cm \big\uparrow
& & \kern -.2cm \big\uparrow
\\
\bullet & \kern -.2cm \longrightarrow & \kern -.2cm \bullet & \kern -.2cm \longrightarrow & \kern -.2cm \bullet & \kern -.2cm \longrightarrow
& \kern -.2cm \bullet & \kern -.2cm \longrightarrow & \kern -.2cm \bullet & \kern -.2cm \longrightarrow & \kern -.2cm \bullet
\\
\big\uparrow & & \kern -.2cm \big\uparrow & & \kern -.2cm \big\uparrow & & \kern -.2cm \big\uparrow & & \kern -.2cm \big\uparrow
& & \kern -.2cm \big\uparrow
\\
\bullet & \kern -.2cm \longrightarrow & \kern -.2cm \bullet & \kern -.2cm \longrightarrow & \kern -.2cm \bullet & \kern -.2cm \longrightarrow
& \kern -.2cm \bullet & \kern -.2cm \longrightarrow & \kern -.2cm \bullet & \kern -.2cm \longrightarrow & \kern -.2cm \bullet
\\
\big\uparrow & & \kern -.2cm \big\uparrow & & \kern -.2cm \big\uparrow & & \kern -.2cm \big\uparrow & & \kern -.2cm \big\uparrow
& & \kern -.2cm \big\uparrow
\\
\bullet & \kern -.2cm \longrightarrow & \kern -.2cm \bullet & \kern -.2cm \longrightarrow & \kern -.2cm \bullet & \kern -.2cm \longrightarrow
& \kern -.2cm \bullet & \kern -.2cm \longrightarrow & \kern -.2cm \bullet & \kern -.2cm \longrightarrow & \kern -.2cm \bullet
\\
\big\uparrow & & \kern -.2cm \big\uparrow & & \kern -.2cm \big\uparrow & & \kern -.2cm \big\uparrow & & \kern -.2cm \big\uparrow
& & \kern -.2cm \big\uparrow
\\
x' & \kern -.2cm \longrightarrow & \kern -.2cm \bullet & \kern -.2cm \longrightarrow & \kern -.2cm \bullet & \kern -.2cm \longrightarrow
& \kern -.2cm \bullet & \kern -.2cm \longrightarrow & \kern -.2cm \bullet & \kern -.2cm \longrightarrow & \kern -.2cm y
\ea
\eeq
where by \eqref{x'x} there exists $(s_2,s_3)\in\RR^2$ such that
\beq\label{x'y'yx}
x':=X_2(t_2',X_1(c_1)),\quad y'=X_3(s_3,x'),\quad y:=X_2(t_2,x'),\quad x=X_3(t_3,y)=X_2(s_2,y').
\eeq
The vertices $x',y',y,x$ of the grid \eqref{Gx'xy'y} satisfy the commutative diagram
\beq\label{X23=X32}
\ba{ccccc}
 y'=X_3(s_3,x') & \dis \stackrel{^{\textstyle X_2}}{\longrightarrow} & X_2(s_2,y')= & \kern -.3cm X_3(t_3,y)=x
\\ \ecart
X_3 \,\big\uparrow & & &  \kern -.2cm \big\uparrow\, X_3
\\ \ecart
x'=X_2(t_2',X_1(c_1)) & \dis \mathop{\longrightarrow}_{\textstyle X_2} & &  \kern -.3cm X_2(t_2,x')=y.
\ea
\eeq
Note that the grid \eqref{Gx'xy'y} is schematic. A more realistic grid is represented in figure~\ref{fig1} below.
\par
Frobenius' theorem will allow us to make the local switching \eqref{compq} thanks to a potential $u$ satisfying $j=\si\nabla u$.
To this end, we proceed by induction on the number $n$ of switchings, for an appropriate time step $\nu>0$ which will be chosen later. The induction hypothesis, for a given $n\in\NN$, consists in the existence of a partial grid $G_n(\nu)$, with $n$ switchings, represented by the diagram
\beq\label{x'phqx}
\ba{ccccccccccc}
& & & & & & & & & & \kern -.2cm x
\\
& & & & & & & & & & \kern -.2cm \big\uparrow
\\
& & & & & & & & & & \kern -.2cm \bullet
\\
& & & & & & & & & & \kern -.2cm \big\uparrow
\\
& & & & & \kern -.2cm \cdots & \kern -.2cm q & \kern -.2cm \longrightarrow & \kern -.2cm \bullet & \kern -.2cm \longrightarrow & \kern -.2cm \bullet
\\
& & & & \kern -.2cm \vdots & & \kern -.2cm \big\uparrow & & \kern -.2cm \big\uparrow & & \kern -.2cm \big\uparrow
\\
\bullet & \kern -.2cm \longrightarrow & \kern -.2cm \bullet & \kern -.2cm \longrightarrow & \kern -.2cm p & \kern -.2cm \longrightarrow
& \kern -.2cm \bullet & \kern -.2cm \longrightarrow & \kern -.2cm \bullet & \kern -.2cm \longrightarrow & \kern -.2cm \bullet
\\*[.1cm]
\kern -.2cm \big\uparrow & & \kern -.2cm \big\uparrow & & \kern -.2cm \big\uparrow & & \kern -.2cm \big\uparrow & & \kern -.2cm \big\uparrow
& & \kern -.2cm \big\uparrow
\\
x' & \kern -.2cm \longrightarrow & \kern -.2cm \bullet & \kern -.2cm \longrightarrow & \kern -.2cm \bullet & \kern -.2cm \longrightarrow
& \kern -.2cm \bullet & \kern -.2cm \longrightarrow & \kern -.2cm \bullet & \kern -.2cm \longrightarrow & \kern -.2cm \bullet
\ea
\eeq
which lies on the surface $\{t_1\!=\!c_1\}$.
\par
First, the result holds for $n=0$. Indeed, the points $x'$ and $x$ defined by \eqref{x'x} clearly belong to the surface $\{t_1\!=\!c_1\}$, so does the initial diagram \eqref{x'hvx} for any time step $\tau$.
\par
Next, assume that after a number $n$ of switchings, we are led to the grid $G_n(\nu)$ \eqref{x'phqx} which, by the induction hypothesis, lies on the surface $\{t_1\!=\!c_1\}$. By virtue of Frobenius' theorem there exist an open neighborhood $V$ of $p$, and a potential $u\in C^1(V)$ such that $j=\si\nabla u$ in $V$. Then, we may chose $\tau>0$ small enough so that
\beq\label{S32S23V}
\ba{ll}
q\in \kern -.2cm & S_{32}(p,\tau):=\big\{X_3(\de_3,X_2(\de_2,p)):|\de_2|+|\de_3|<\tau\big\}\subset V
\\ \ecart
& S_{23}(p,\tau):=\big\{X_2(\de_2,X_3(\de_3,p)):|\de_2|+|\de_3|<\tau\big\}\subset V,
\ea
\eeq
independently of the point $p$ in a given compact set of $\RR^3$. Since the potential $u$ is constant along the flows $X_2$ and $X_3$, we have for any point $r\in S_{32}(p,\tau)$,
\[
u(r)=u\big(X_3(\de_3,X_2(\de_2,p))\big)=u\big(X_3(0,X_2(\de_2,p))\big)=u\big(X_2(\de_2,p)\big)=u\big(X_2(0,p)\big)=u(p).
\]
The same equality holds for any point $r\in S_{23}(p,\tau)$, hence
\beq\label{S32S23u=up}
q\in S_{32}(p,\tau)\subset\big\{u=u(p)\big\}\cap V\quad\mbox{and}\quad S_{23}(p,\tau)\subset\big\{u=u(p)\big\}\cap V.
\eeq
Moreover, since the trajectory along the flow $X_2$, passing through the point $p$ in diagram \eqref{x'phqx} lies on the surface $\{t_1\!=\!c_1\}$ (by the induction hypothesis), thanks to the semi-group property satisfied by the flow $X_3$, we have that for any $|\de_2|+|\de_3|<t$, there exists $(t_2,t_3)\in\RR^2$ such that
\beq
X_3\big(\de_3,X_2(\de_2,p)\big)=X_3\big(\de_3,X_{32}(t_3,t_2,c_1)\big)=X_{32}(\de_3+t_3,t_2,c_1)\in \big\{t_1\!=\!c_1\big\}.
\eeq
Hence, by the definition \eqref{S32S23V} of $S_{32}(p,\tau)$, we get that
\beq\label{S32t1=c1}
\forall\,t>0,\quad S_{32}(p,t)\subset\big\{t_1\!=\!c_1\big\}.
\eeq
However, thanks to the condition \eqref{ortbasj} which yields $\nabla u\neq 0$ in $V$, the sets $S_{32}(p,\tau)$, $S_{23}(p,\tau)$ and $\{u=u(p)\}\cap V$
%,\quad \{t_1=c_1\big\}\cap V
are regular open surfaces in $V$. Hence, from the inclusions \eqref{S32S23u=up} we deduce that the surfaces $S_{32}(p,\tau)$ and $S_{23}(p,\tau)$ actually agree with the equipotential $\{u=u(p)\}$ in some neighborhood of $p$ containing $q$, but independent of $p$ in a given compact set $K$ of $\RR^3$. This combined with the condition \eqref{S32t1=c1} (which does not depend on time), implies that the time step $\nu<\tau$  may be chosen small enough, but independently of the point $p$ in $K$, so that
\beq
q\in S_{23}(p,\nu)\subset\big\{t_1\!=\!c_1\big\}.
\eeq
Therefore, the following grid $G_{n+1}(\nu)$ which completes \eqref{x'phqx}
\beq\label{xpvqx'}
\ba{ccccccccccc}
& & & & & & & & & & \kern -.2cm x
\\
& & & & & & & & & & \kern -.2cm \big\uparrow
\\
& & & & & & & & & & \kern -.2cm \bullet
\\
& & & & & & & & & & \kern -.2cm \big\uparrow
\\
& & & &  \kern -.2cm \bullet & \kern -.2cm \longrightarrow & \kern -.2cm q & \kern -.2cm \longrightarrow & \kern -.2cm \bullet & \kern -.2cm \longrightarrow & \kern -.2cm \bullet
\\
& & & & \kern -.2cm \big\uparrow  & & \kern -.2cm \big\uparrow & & \kern -.2cm \big\uparrow & & \kern -.2cm \big\uparrow
\\
\bullet & \kern -.2cm \longrightarrow & \kern -.2cm \bullet & \kern -.2cm \longrightarrow & \kern -.2cm p & \kern -.2cm \longrightarrow
& \kern -.2cm \bullet & \kern -.2cm \longrightarrow & \kern -.2cm \bullet & \kern -.2cm \longrightarrow & \kern -.2cm \bullet
\\*[.1cm]
\kern -.2cm \big\uparrow & & \kern -.2cm \big\uparrow & & \kern -.2cm \big\uparrow & & \kern -.2cm \big\uparrow & & \kern -.2cm \big\uparrow
& & \kern -.2cm \big\uparrow
\\
x' & \kern -.2cm \longrightarrow & \kern -.2cm \bullet & \kern -.2cm \longrightarrow & \kern -.2cm \bullet & \kern -.2cm \longrightarrow
& \kern -.2cm \bullet & \kern -.2cm \longrightarrow & \kern -.2cm \bullet & \kern -.2cm \longrightarrow & \kern -.2cm \bullet
\ea
\eeq
also lies on the surface $\{t_1=c_1\}$.
The induction proof is thus done, which establishes the existence of the grid \eqref{Gx'xy'y}.
\par\ms\noindent
{\it Third step: } Proof of the isotropic realizability with the conductivity $e^w$.
\par\ss\noindent
Let us prove that the function $w$ defined by \eqref{w(x)} combined with condition \eqref{diffX32} satisfies the equality \eqref{jxcurljDw}. This yields the global isotropic realizability of $j$, since by \eqref{divj=0}, \eqref{j.curlj=0} and \eqref{jxcurljDw} we have
\beq\label{curlewj}
\curl(e^{-w}j)=e^{-w}\left(\curl j-\nabla w\times j\right)=e^{-w}\left(\curl j-{(j\times\curl j)\times j\over|j|^2}\right)=0.
\eeq
\par
On the one hand, taking the derivative of \eqref{w(x)} with respect to $t_3$, we have
\beq
\nabla w\cdot\left({j\times\curl j\over|j|^2}\right)\big(X_{32}(t_3,t_2,t_1)\big)={|\curl j|^2\over|j|^2}\big(X_{32}(t_3,t_2,t_1)\big),\quad\forall\,(t_1,t_2,t_3)\in\RR^3,
\eeq
which together with condition \eqref{diffX32} implies that
\beq\label{Dwjxcurlj}
\nabla w\cdot\left({j\times\curl j\over|j|^2}\right)={|\curl j|^2\over|j|^2}=\left|{j\times\curl j\over|j|^2}\right|^2\quad\mbox{in }\RR^3.
\eeq
Moreover, taking the derivative of \eqref{w(x)} with respect to $t_2$ for $t_3=0$, we get that
\[
\ba{ll}
\dis 0=\nabla w\big(X_{32}(0,t_2,t_1)\big)\cdot{\partial X_{32}\over\partial t_2}(0,t_2,t_1) & \dis =\nabla w\big(X_2(t_2,X_1(t_1))\big)\cdot
{\partial X_2\over\partial s}(t_2,X_1(t_1))
\\ \ecart
& \dis =\nabla w\big(X_2(t_2,X_1(t_1))\big)\cdot{\curl j\over|\curl j|}\big(X_2(t_2,X_1(t_1))\big),
\ea
\]
%where we have used that $\partial_{t_2}X_3$ satisfies the linear ode
%\[\partial_{t}\left(\partial_{t_2}X_{32}(t,t_2,t_1)\right)=D\big({\tex {j\times{\rm curl}\,j\over|j|^2}}\big)\big(X_{32}(t,t_2,t_1)\big)\,\partial_{t_2}X_{32}(t,t_2,t_1),\quad \partial_{t_2}X_{32}(0,t_2,t_1)=\partial_{t_2}X_2(t_2,X_1(t_1)).\]
which yields
\beq\label{Dwcurljt3=0}
\nabla w\cdot\curl j=0\quad\mbox{on }\big\{X_2(t_2,X_1(t_1)):(t_1,t_2)\in\RR^2\big\}.
\eeq

\begin{figure}[t!]
\centering
\includegraphics[scale=.4]{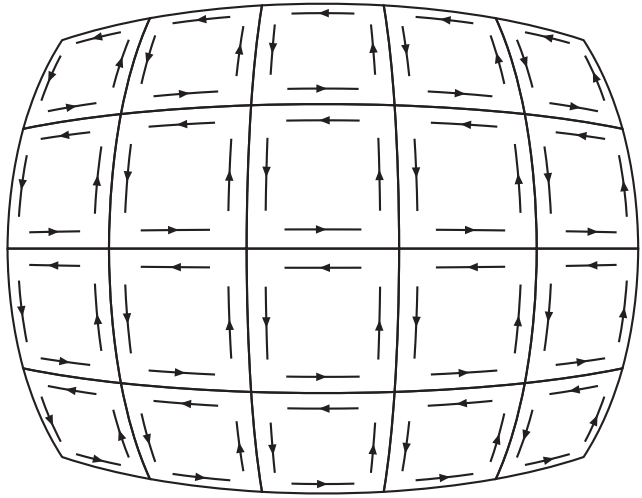}
\caption{\it The square $Q$ split in small squares $Q_k$ on the surface $\{t_1=c_1\}$}
\label{fig1}
\end{figure}

\par
On the other hand, for a constant $c_1\in\RR$, consider on the surface $\{t_1=c_1\}$ the curvilinear integral over $\partial Q$ defined by
\beq\label{intQdom}
\int_{\stackrel{\curvearrowleft}{\partial Q}} {\om}:=\int_{0}^{t_3}{|\curl j|^2\over|j|^2}(X_3(s,y)\big)\,ds
-\int_{0}^{s_3}{|\curl j|^2\over|j|^2}\big(X_3(s,x')\big)\,ds,
\eeq
where $Q=Q_{x',y',y,x}$, with $x',y',y,x$ defined by \eqref{x'y'yx}, is the ``square" lying on the surface $\{t_1=c_1\}$, whose:
\begin{itemize}
\item horizontal sides are trajectories associated with the flow $X_2$ along $\curl j/|\curl j|$,
\item vertical sides are trajectories associated with the flow $X_3$ along $j/|j|^2\times\curl j$,
\item vertices $x',y',y,x$ make the loop \eqref{X23=X32}.
\end{itemize}
Now, consider the thin grid on $Q$,   associated with the grid \eqref{Gx'xy'y}, whose lines are parallel to the trajectories along $\curl j/|\curl j|$ and $(j/|j|^2)\times\curl j$, and which splits $Q$ in small squares $Q_k$, as shown in \reff{fig1}.. According to the second step of the proof, we may choose the time step of the grid so that Frobenius' theorem applies in some open neighborhood $V_k$ of $\RR^3$, containing the closed square $\overline{Q_k}$. Namely, there exists $w_k\in C^1(V_k)$ such that
\[
\curl(e^{-w_k}\,j)=0\quad\mbox{in }V_k,
\]
which implies formulas \eqref{jxcurljDw} and \eqref{Dwjxcurlj} for $w_k$. Using the loops induced by the boundaries $\partial Q_k$, the contributions of the interior vertical lines (along $j/|j|^2\times\curl j$) two by two cancel (see \reff{fig1}.), which leads to
\beq\label{intomQQk}
\int_{\stackrel{\curvearrowleft}{\partial Q}} {\om}=\sum_{k}\int_{\stackrel{\curvearrowleft}{\partial Q_k}} {\om}.
\eeq
However, since
\[
\nabla w_k\cdot\curl j=0\quad\mbox{in }V_k\supset\overline{Q_k},
\]
the curvilinear integrals of $dw_k$ on the two horizontal lines of $\partial Q_k$ (along $\curl j/|\curl j|$) are equal to~$0$, while by putting \eqref{Dwjxcurlj} in \eqref{intQdom} the curvilinear integrals of $dw_k$ on the two vertical lines of $\partial Q_k$ (along $j/|j|^2\times\curl j$) agree with the curvilinear integral of ${\om}$ over $\partial Q_k$, which yields
\beq\label{intomQk=0}
\int_{\stackrel{\curvearrowleft}{\partial Q_k}} {\om}=\int_{\stackrel{\curvearrowleft}{\partial Q_k}} dw_k=0.
\eeq
The last equality is due to the fact that $dw_k$ is an exact differential on the closed loop~$\partial Q_k$.
Hence, from \eqref{intomQQk} and \eqref{intomQk=0} we deduce that
\beq\label{intQom=0}
\int_{\stackrel{\curvearrowleft}{\partial Q}} {\om}=0.
\eeq
\par
Next, consider the function $w$ defined by \eqref{w(x)}.
Taking the loop \eqref{X23=X32} in the anti-clockwise direction, the contributions of the vertical lines 
(along $(j/|j|^2)\times\curl j$) of the curvilinear integral of $dw$ over $\partial Q$ in the anti-clockwise direction, read as, in view of \eqref{Dwjxcurlj} and \eqref{intQdom},
\beq\label{verlin}
\int_{0}^{t_3}\nabla w\cdot\!\left({j\times\curl j\over|j|^2}\right)\!\big(X_3(s,y)\big)\,ds
-\int_{0}^{s_3}\nabla w\cdot\!\left({j\times\curl j\over|j|^2}\right)\!\big(X_3(s,x')\big)\,ds
=\int_{\stackrel{\curvearrowleft}{\partial Q}} {\om}.
\eeq
Moreover, again by \eqref{X23=X32} the contributions of the horizontal lines (along $\curl j/|\curl j|$) of the curvilinear integral of $dw$ over $\partial Q$ in the anti-clockwise direction, read as
\beq\label{horlin}
\int_{0}^{t_2}\nabla w\cdot{\curl j\over|\curl j|}\big(X_2(t_2,x')\big)\,ds
-\int_{0}^{s_2}\nabla w\cdot{\curl j\over|\curl j|}\big(X_2(s,y')\big)\,ds.
\eeq
From \eqref{verlin} and \eqref{horlin} we deduce that
\[
\ba{ll}
\dis 0 & \dis =\int_{\stackrel{\curvearrowleft}{\partial Q}} dw
\\ \ecart
& \dis =\int_{\stackrel{\curvearrowleft}{\partial Q}} {\om}
+\int_{0}^{t_2}\nabla w\cdot{\curl j\over|\curl j|}\big(X_2(t_2,x')\big)\,ds
-\int_{0}^{s_2}\nabla w\cdot{\curl j\over|\curl j|}\big(X_2(s,y')\big)\,ds,
\ea
\]
which by \eqref{intQom=0} implies that
\beq\label{intDwcurlj}
\int_{0}^{s_2}\nabla w\cdot{\curl j\over|\curl j|}\big(X_2(s,y')\big)\,ds
=\int_{0}^{t_2}\nabla w\cdot{\curl j\over|\curl j|}\big(X_2(t_2,x')\big)\,ds.
\eeq
Equality \eqref{intDwcurlj} combined with \eqref{x'y'yx} and \eqref{Dwcurljt3=0} yields
\beq\label{intDwcurlj2}
{1\over s_2}\int_{0}^{s_2}\nabla w\cdot{\curl j\over|\curl j|}\big(X_2(s,X_3(s_3,x'))\big)\,ds=0,\quad\mbox{where}\quad x'=X_2(t_2',X_1(c_1)).
\eeq
Then, making $t_2\to 0$ -- which, by the continuity of the flows in the diagram \eqref{X23=X32}, implies that $s_2\to 0$ and $X_3(s_3,x')\to X_3(t_3,x')$ -- and using the continuity of the integrand in the left-hand side of~\eqref{intDwcurlj2}, we get that
\beq
\nabla w\cdot\curl j\,\big(X_{32}(t_3,t_2',c_1)\big)=\nabla w\cdot\curl j\,\big(X_3(t_3,x')\big)=0,\quad\forall\,(c_1,t_2',t_3)\in\RR^3.
\eeq
This combined with the diffeomorphism condition \eqref{diffX32} leads to
\beq\label{Dw.curlj=0}
\nabla w\cdot\curl j=0\quad\mbox{in }\RR^3.
\eeq
\par
Finally, combining the equalities \eqref{Dw.curlj=0} and \eqref{Dwjxcurlj} with \eqref{ortbasj}, we derive the formula \eqref{jxcurljDw}.
Therefore, thanks to \eqref{curlewj} we may conclude that the current field $j$ is isotropically realizable with the conductivity $e^w$ defined by \eqref{w(x)}, which concludes the proof of Theorem~\ref{thm.glorea}.
%\beq\label{w(x)}\si(x)=\si\big(X_{32}(t_3,t_2,t_1)\big):=\exp\left(\int_0^{t_3}{|\curl j|^2\over|j|^2}\big(X_{32}(s,t_2,t_1)\big)\,ds\right),\quad\mbox{for }x\in\RR^3.\eeq
\cqfd
\subsection{Isotropic realizability in the torus}
As in the case of an electric field \cite{BMT} the isotropic realizability of a periodic current field in $\RR^3$ does not imply in general the isotropic realizability in the torus, as shown in Example~\ref{rem.f} below.
\par
First of all, we have the following criterion of isotropic realizability in the torus:
\begin{Pro}\label{pro.perrea}
Let $j$ be a $Y$-periodic divergence free field in $L^\infty_\sharp(Y)^3$. Then, the field $j$ is isotropically realizable in the torus with a conductivity $\si>0$ satisfying $\si,\si^{-1}\in L^\infty_\sharp(Y)$, if and only if there exists $w\in L^\infty(\RR)$ such that
\beq\label{curlwj}
\curl(e^{-w}j)=0\quad\mbox{in }\RR^3.
\eeq
\end{Pro}
\begin{proof}
Assume that $j$ is isotropically realizable in the torus with a conductivity $\si>0$ satisfying $\si,\si^{-1}\in L^\infty_\sharp(Y)$.
Then, defining $w:=\ln\si\in L^\infty(\RR^3)$, the function $e^{-w}j$ is a gradient, which implies \eqref{curlwj}.
\par
Conversely, assume that \eqref{curlwj} holds with $w\in L^\infty(\RR)$. Then, there exists $u\in W^{1,\infty}(\RR^3)$ such that $e^{-w}j=\nabla u$, or equivalently $j=e^w\,\nabla u$ in $\RR^3$. It is not clear that $e^w$ is periodic. However, we can construct a suitable periodic conductivity by adapting the average argument of \cite{BMT} (proof of Theorem~2.17).
To this end, define the sequence $(w_n)_{n\in\NN\setminus\{0\}}$, by
\[
w_n(x):={1\over (2n+1)^3}\kern -.2cm\sum_{k\in\ZZ^3,\,|k|_\infty\leq n}\kern -.4cm w(x+k),
\quad\mbox{for }x\in\RR^3,\quad\mbox{where }|k|_\infty:=\max\left(|k_1|,|k_2|,|k_3|\right).
\]
Since $w$ is in $L^\infty(\RR^3)$, the sequence $w_n$ is bounded in $L^\infty(\RR)$, and thus converges weakly-$*$ to some function $w_\sharp$ in $L^\infty(\RR^3)$. It is easy to check that $w_\sharp$ is $Y$-periodic. Moreover, by the periodicity of $j$ we have
\[
\curl(e^{-w_n}j)={1\over (2n+1)^3}\kern -.2cm\sum_{k\in\ZZ^3,\,|k|_\infty\leq n}\kern -.4cm\curl\big(e^{-w(\cdot+k)}j(\cdot+k)\big)
={1\over (2n+1)^3}\kern -.2cm\sum_{k\in\ZZ^3,\,|k|_\infty\leq n}\kern -.4cm\curl\big(\nabla u(\cdot+k)\big)=0.
\]
As $e^{-w_n}j$ converges weakly-$*$ to $e^{-w_\sharp}j$ in $L^\infty(\RR^3)^3$, the previous equality leads to
\[
\curl(e^{-w_\sharp}j)=0\quad\mbox{in }\D'(\RR^3)^3.
\]
Therefore, $e^{-w_\sharp}j$ is a periodic gradient and $j$ is isotropically realizable in the torus with the conductivity $e^{w_\sharp}$.
\end{proof}
As a consequence of Proposition~\ref{pro.perrea} and Theorem~\ref{thm.glorea}, we have the following result:
\begin{Cor}\label{cor.perrea}
Let $j$ be a $Y$-periodic current field satisfying conditions \eqref{divj=0} and \eqref{j.curlj=0}.
\begin{itemize}
\item[$i)$] Assume that the field $j$ is isotropically realizable in the torus with a positive conductivity $\si\in C^1_\sharp(Y)$.
Then, there exists a constant $C>0$ such that
\beq\label{L2X3'}
\int_{-\infty}^{\infty}{|\curl j|^2\over |j|^2}\big(X_3(s,x)\big)\,ds\leq C,\quad\forall\,x\in\RR^3,
\eeq
where the flow $X_3$ is defined by \eqref{Xi}.
\item[$ii)$] Alternatively, assume that the conditions \eqref{ortbasj} and \eqref{diffX32} are satisfied.
Then, the field $j$ is isotropically realizable in the torus if and only if estimate \eqref{L2X3'} holds true.
\end{itemize}
\end{Cor}
\begin{Rem}
Corollary~\ref{cor.perrea} with the boundedness condition~\eqref{L2X3'} is illustrated in Proposition~\ref{pro.fgh} below.
\end{Rem}
\par\noindent
{\bf Proof of Corollary~\ref{cor.perrea}.}
\par\ss\noindent
{\it Proof of $i)$.}
Denoting $\si=e^w\in C^1(\RR^3)$, we have $j=e^w\nabla u$, with $\nabla u\in C^1(\RR^3)^3$.
Then, from \eqref{jxcurljDw} we deduce that
\[
{\partial\over\partial t}\left[w\big(X_3(t,x)\big)\right]={|\curl j|^2\over |j|^2}\big(X_3(t,x)\big),\quad\forall\,(t,x)\in\RR\times\RR^3,
\]
which yields
\beq\label{wdX}
w\big(X_3(t,x)\big)-w(x)=\int_0^{t}{|\curl j|^2\over |j|^2}\big(X_3(s,x)\big)\,ds,\quad\forall\,(t,x)\in\RR\times\RR^3.
\eeq
Therefore, due to the boundedness of $w\in C^1_\sharp(Y)$, the estimate \eqref{L2X3'} holds.
\par\ms\noindent
{\it Proof of $ii)$.} By Theorem~\ref{thm.glorea} we already know that $j$ is isotropically realizable with the conductivity $\si=e^w\in C^1(\RR^3)$ defined by \eqref{w(x)}. If the field $j$ is isotropically realizable in the torus, then by $i)$ the estimate \eqref{L2X3'} holds. Conversely, in view of the estimate \eqref{L2X3'} combined with the definition \eqref{X1X23} of $X_{32}$, the function $w$ defined by \eqref{w(x)} clearly belongs to $L^\infty(\RR^3)$. Therefore, applying Proposition~\ref{pro.perrea} with the conductivity $e^w$, we get that $j$ is isotropically realizable in the torus.
\cqfd
\section{Examples and counter-examples}\label{s.exa}
In this section a point of $\RR^3$ is denoted by the coordinates $(x,y,z)$, and $(e_x,e_y,e_z)$ denotes the canonical basis of $\RR^3$.
\subsection{Example of global realizability in the space}\label{ss.jsh}
We will illustrate the construction of Theorem~\ref{thm.glorea} with the current field
\beq\label{jsh}
j(x,y,z):=\begin{pmatrix} 1 \\ \sinh x \\ 0 \end{pmatrix},\quad\mbox{for }(x,y,z)\in\RR^3,
\eeq
which is clearly non-zero and divergence free in $\RR^3$. We have
\beq
\curl j(x,y,z):=\begin{pmatrix} 0 \\ 0 \\ \cosh x \end{pmatrix}\quad\mbox{and}\quad
j\times \curl j(x,y,z)=\begin{pmatrix} \sinh x\,\cosh x \\ -\cosh x \\ 0 \end{pmatrix},
\eeq
so that the field $j$ satisfies \eqref{divj=0} and Frobenius' condition \eqref{j.curlj=0}.
\par
Noting that $|j|=|\curl j|=|\cosh x|$, and using the solutions to one-dimensional first-order odes, the flows $X_1$, $X_2$, $X_3$ of \eqref{Xi} associated with the field $j$ are given by
\beq
\ba{c}
\dis X_1(t)=\begin{pmatrix} {\rm argsinh}\,\big(t+\sinh\,(x_1(0))\big) \\ \sqrt{1+\big(t+\sinh\,(x_1(0))\big)^2}+y_1(0)-\cosh(x_1(0)) \\ z_1(0) \end{pmatrix},\quad
X_2(t)=\begin{pmatrix} x_2(0) \\ y_2(0) \\ t+z_2(0) \end{pmatrix},
\\ \ecart
X_3(t)=\begin{pmatrix} {\rm argsinh}\,\big(e^t\sinh\,(x_3(0))\big) \\ \dis -\int_0^{t}{ds\over\sqrt{1+e^{2s}\sinh^2(x_3(0))}}+y_3(0) \\ z_3(0) \end{pmatrix}.
\ea
\eeq
Hence, starting from the point $x_0=(0,1,0)$, the composed flows $X_{32}$, $X_{23}$ of \eqref{X1X23} are given by
\beq\label{X23sh}
X_{32}(t_3,t_2,t_1)=
\begin{pmatrix} {\rm argsinh}\,\big(t_1\,e^{t_3}\big) \\ \dis -\int_0^{t_3}{ds\over\sqrt{1+e^{2s}\,t_1^2}}+\sqrt{1+t_1^2} \\ t_2 \end{pmatrix},
\quad\mbox{for }(t_1,t_2,t_3)\in\RR^3.
\eeq
Making classical changes of variables in the integral term, formula \eqref{X23sh} can be written
\beq\label{X23shf}
X_{32}(t_3,t_2,t_1)=\left\{\ba{cl}
\begin{pmatrix} 0 \\ \dis 1-t_3 \\ t_2 \end{pmatrix}& \mbox{if }t_1=0
\\\\
\begin{pmatrix} {\rm argsinh}\,\big(t_1\,e^{t_3}\big) \\ \dis {\rm argtanh}\left({1\over\sqrt{1+t_1^2\,e^{2t_3}}}\right)+f(t_1) \\ t_2 \end{pmatrix}
& \mbox{if }t_1\neq 0,
\ea\right.
\eeq
where
\beq\label{shf}
f(t):=-\,{\rm argtanh}\left({1\over\sqrt{1+t^2}}\right)+\sqrt{1+t^2},\quad\mbox{for }t\neq 0.
\eeq
The function $f$ satisfies
\beq\label{shff'}
f(t)=\ln|t|+1-\ln 2+o(1)\quad\mbox{and}\quad f'(t)={\sqrt{1+t^2}\over t}={1\over t}+o(1),
\eeq
and is a $C^1$-diffeomorphism from $(0,\infty)$ onto $\RR$, and from $(-\infty,0)$ onto $\RR$.
Hence, we have
\beq\label{fdiff}
\forall\,(x,y)\in\RR\setminus\{0\}\times\RR,\ \exists\,!\,t_1\in\RR,\quad x\,t_1>0\quad\mbox{and}\quad y-{\rm argtanh}\left({1\over\cosh x}\right)=f(t_1),
\eeq
which together with \eqref{X23shf} implies that for this $t_1$,
\beq\label{xyzX32}
(x,y,z)=X_{32}(t_3,t_2,t_1)\quad\Leftrightarrow\quad
\left\{\ba{ll}
\left\{\ba{l}
x=0
\\
\dis y=1-t_3
\\
z=t_2
\ea\right.
& \mbox{if }t_1=0
\\ \ecart
\left\{\ba{l}
\sinh x=t_1\,e^{t_3}
\\
\dis y-{\rm argtanh}\left({1\over\cosh x}\right)=f(t_1)
\\
z=t_2.
\ea\right.
& \mbox{if }t_1\neq 0.
\ea\right.
\eeq
As a consequence of \eqref{X23shf}, \eqref{shf}, the first equality of \eqref{shff'}, \eqref{fdiff} and \eqref{xyzX32}, the mapping $X_{32}$ define a homeomorphism onto $\RR^3$, which is of class $C^1$ by \eqref{X23sh}. Unhappily, we can check that the Jacobian of $X_{23}$ vanishes (exclusively) on the line $\{t_1=0\}$. However, taking into account the equality $\sinh x=t_1\,e^{t_3}$, $X_{23}$ does establish a $C^1$-diffeomorphism from the half-spaces $\{\pm\,t_1>0\}$ onto the half-spaces $\{\pm\,x>0\}$. Therefore, condition \eqref{diffX32} holds true restricting ourselves on these half-spaces.
\par
On the other hand, since $|j|=|\curl j|$, the function $w$ defined by \eqref{w(x)} reads as
\beq\label{wsh}
w(x,y,z)=t_3=\left\{\ba{lll}
1-y & \mbox{if }x=0 &
\\ \ecart
\dis \ln\left({\sinh x\over t}\right) & \mbox{if }x\neq 0, & \dis \mbox{where}\quad x\,t>0,\ y-{\rm argtanh}\left({1\over\cosh x}\right)=f(t).
\ea\right.
\eeq
It is easy to check that the asymptotic expansions \eqref{shff'} satisfied by $f$ imply that $w\in C^1(\RR^3)$. Hence, the conditions of Theorem~\ref{thm.glorea} are fulfilled in the two half-spaces $\{\pm\,x>0\}$. Therefore, the field $j$ defined by \eqref{jsh} is isotropically realizable in the half-spaces $\{\pm\,x>0\}$, with the conductivity $\si\in C^1(\RR^3)$ given by
\beq\label{sish}
\si(x,y,z):=\left\{\ba{lll}
e^{1-y} & \mbox{if }x=0 &
\\ \ecart
\dis {\sinh x\over t} & \mbox{if }x\neq 0, & \dis \mbox{where}\quad x\,t>0,\ y-{\rm argtanh}\left({1\over\cosh x}\right)=f(t).
\ea\right.
\eeq
Finally, the $C^1$-regularity of $j$ and $\si$ ensure the isotropic realizability in the whole space $\RR^3$.
\par
Since $X_{32}$ is a homeomorphism onto $\RR^3$ of class $C^1$ and the function $w$ of \eqref{wsh} is in $C^1(\RR^3)$, we can also conclude  thanks to Remark~\ref{rem.homX32}.
\begin{Rem}\label{rem.jsh}
The previous example allows us to show that there exists an infinite set of admissible conductivities which cannot be derived from a multiple of the conductivity \eqref{sish} defined with the function \eqref{shf}. To this end, consider any function $f:\RR\setminus\{0\}\to\RR$ which is a $C^1$-diffeomorphism from $(0,\infty)$ onto $\RR$ and from $(-\infty,0)$ onto $\RR$, which satisfies the following asymptotic expansions around $0$:
\beq\label{expff'}
f(t)=\ln|t|+c+o(1)\quad\mbox{and}\quad f'(t)={1\over t}+c'+o(1),\quad\mbox{for some }c,c'\in\RR.
\eeq
Then, define the conductivity $\si_f$ by
\beq\label{sif}
\si_f(x,y,z):=\left\{\ba{lll}
2\,e^{c-y} & \mbox{if }x=0 &
\\ \ecart
\dis {\sinh x\over t} & \mbox{if }x\neq 0, & \dis \mbox{where}\quad x\,t>0,\ y-{\rm argtanh}\left({1\over\cosh x}\right)=f(t).
\ea\right.
\eeq
Thanks to \eqref{shf}, \eqref{shff'} and \eqref{expff'} the function $\si_f$ belongs to $C^1(\RR)$.
Moreover, we have for any $(x,y)\in\RR\setminus\{0\}\times\RR$,
\beq
\curl\big(\si_f^{-1}\,j\big)=\curl\begin{pmatrix} \dis {t\over\sinh x} \\ t \\ 0 \end{pmatrix}
=\begin{pmatrix} 0 \\ 0 \\ \dis {\partial t\over\partial x}-{1\over\sinh x}\,{\partial t\over\partial y} \end{pmatrix}.
\eeq
This combined with
\beq
-\,{\partial\over\partial x}\left[{\rm argtanh}\left({1\over\cosh x}\right)\right]={1\over \sinh x}={\partial t\over\partial x}\,f'(t)
\quad\mbox{and}\quad1={\partial t\over\partial y}\,f'(t),
\eeq
yields that $\curl(\si_f^{-1}j)=0$ in $\RR\setminus\{0\}\times\RR$. Since $\si_f$ is in $C^1(\RR)$, the equality holds in $\RR^2$.
Therefore, the field $j$ defined by \eqref{jsh} is isotropically realizable in $\RR^3$ with any conductivity $\si_f$ defined by \eqref{sif} and \eqref{expff'}.
\end{Rem}
\subsection{Example of non-global realizability under Frobenius' condition}\label{ss.cexF}
The following example shows that condition \eqref{j.curlj=0} is not sufficient to derive a global realizability result in accordance with the local character of Frobenius' theorem. This example is an extension to a divergence free non-vanishing field in $\RR^3$ of the counterexample of \cite{Boy} obtained for a non-vanishing field in $\RR^2$.
\par
Define the function $j$ in $\RR^3$ by
\beq\label{jcex}
j(x,y,z):=\begin{pmatrix} f(x) \\ f'(x) \\ -z\,f'(x)\end{pmatrix},\quad\mbox{where}\quad f(x):=8\,x^3-6\,x^4-1,
\eeq
which satisfies condition \eqref{divj=0}.
We have
\[
\curl j(x,y,z)=\begin{pmatrix} 0 \\ z\,f''(x) \\ f''(x) \end{pmatrix},
\]
so that Frobenius' condition \eqref{j.curlj=0} holds, but not \eqref{ortbasj} since $f''(\sqrt{2/3})=0$.
\par
Now, assume that $j$ can be written $\sigma\nabla u$ for some positive continuous function $\sigma$ in the closed strip $\{0\leq x\leq 1\}$.
We have for any $x\in(0,1)$, $f'(x)\neq 0$ and
\beq\label{fu}
\left\{\ba{l}
\dis f(x)=\sigma\,{\partial u\over\partial x}
\\ \ecart
\dis f'(x)=\sigma\,{\partial u\over\partial y},
\ea\right.
\quad\mbox{hence}\quad
{\partial u\over\partial x}={f(x)\over f'(x)}\,{\partial u\over\partial y}.
\eeq
Then, the method of characteristics implies that for a fixed $z\in\RR$, there exists a function $g_z$ in $C^1(\RR)$ such that
\beq\label{ugz}
u(x,y,z)=g_z\big(y+F(x)\big),\quad\forall\,(x,y)\in(0,1)\times\RR,
\eeq
where $F$ is the primitive of $f/f'$ on $(0,1)$ defined by
\beq\label{F}
F(x):={x^2\over 8}-{x\over 12}-{\ln(1-x)\over 24}-{\ln x\over 24}+{1\over 24x},\quad\mbox{for }x\in(0,1).
\eeq
As $x\to 0$, $x>0$, we have by \eqref{F}
\[
t:=y+F(x)\,\mathop{\sim}_{x\to 0}{1\over 24\,x}.
\]
This combined with \eqref{jcex}, \eqref{fu}, \eqref{ugz} yields
\[
g'_z(t)={1\over\sigma(x,y,z)}\,f'(x)\,\mathop{\sim}_{t\to\infty}\,{1\over 24\,\sigma(0,y,z)}\,{1\over t^2}.
\]
As $x\to 1$, $x>1$, we have by \eqref{F}
\[
t=y+F(x)=y+{1\over 12}-{\ln(1-x)\over 24}+o(1),\quad\mbox{hence}\quad 1-x \,\mathop{\sim}_{t\to\infty}\,e^{24y+2}\,e^{-24t},
\]
which implies that
\[
g'_z(t)={1\over\sigma(x,y,z)}\,f'(x)\,\mathop{\sim}_{t\to\infty}\,{24\,e^{24y+2}\over\sigma(1,y,z)}\,e^{-24t}.
\]
Therefore, the two asymptotic expansions of $g_z(t)$ as $t\to\infty$, lead to a contradiction.
\subsection{The case where the current field lies in a fixed plane}\label{ss.jplan}
Consider a periodic field $j$ in $C^2_\sharp(Y)^3$ satisfying \eqref{divj=0} and \eqref{j.curlj=0} which remains perpendicular to a fixed direction. By an orthogonal change of variables we may assume that $j$ lies in the plane $\{z=0\}$, namely $j=(j_x,j_y,0)$.
From now on, any vector of $\RR^2$ will be identified to a vector of~$\RR^3$ with zero $z$-coordinate.
Hence, we deduce that
\beq\label{j.curlj=02d}
j\cdot\curl j=0\;\Leftrightarrow\;\partial_z j\parallel j\quad\mbox{in }\RR^3.
\eeq
Then, using the representation of a divergence-free field by an orthogonal gradient in $\RR^2$, the most general expression for a divergence free field $j$ satisfying \eqref{j.curlj=02d} is
\beq\label{jF2d}
\ba{rl}
j(x,y,z)=\al(x,y,z)\,\nabla^\perp v(x,y)=\al(x,y,z)\begin{pmatrix} -\,\partial_y v(x,y) \\ \partial_x v(x,y) \\ 0 \end{pmatrix}
& \mbox{for }(x,y,z)\in\RR^3,
\\ \ecart
\mbox{with}\quad \partial_x\al\,\partial_y v-\partial_y\al\,\partial_x v=0 & \mbox{in }\RR^3,
\ea
\eeq
where $\al\in C^2_\sharp(Y')$, $\al>0$, and $v\in C^3(\RR^2)^2$, with $\nabla v$ $Y'$-periodic.
For the sake of simplicity, we assume from now on that the function $\al$ only depends on the coordinate $z$. Therefore, we are led to
\beq\label{jz=0}
j(x,y,z)=\al(z)\,\nabla^\perp v(x,y),\quad \mbox{for }(x,y,z)\in\RR^3,
\eeq
where $\al\in C^2_\sharp([0,1])$, $\al>0$, and $v\in C^3(\RR^2)^2$, with $\nabla v$ $Y'$-periodic and $\nabla v\neq 0$ in $\RR^2$ by \eqref{divj=0}.
\par
Following the isotropic realizability procedure for a gradient field \cite{BMT}, consider the gradient system
\beq\label{Z}
\left\{\ba{ll}
Z'(t,x,y)=\nabla v\big(Z(t,x,y)\big), & \mbox{for }t\in\RR,
\\ \ecart
Z(0,x,y)=(x,y)\in\RR^2. &
\ea\right.
\eeq
Then, by virtue of Theorem~\ref{thm.elerea} there exists a unique function $\tau_v\in C^1(\RR^2)$ such that
\beq\label{tauv}
v\big(Z(\tau_v(x,y),x,y)\big)=0,\quad\mbox{for any }(x,y)\in\RR^2,
\eeq
and the function $w_v$ defined by
\beq\label{wv}
w_v(x,y):=-\int_0^{\tau_v(x,y)}\De v\big(Z(s,x,y)\big)\,ds,\quad\mbox{for }(x,y)\in\RR^2,
\eeq
satisfies the conductivity equation $\div(e^{-w_v}\nabla v)=0$ in $\RR^2$.
\par
We have the following isotropic realizability result:
\begin{Pro}\label{pro.jz=0}
If the function $w_v$ of \eqref{wv} is in $L^\infty(\RR^2)$, then the field $j$ defined by \eqref{jz=0} is isotropically realizable in the torus. Conversely, if $j$ is isotropic realizable in the torus with a positive conductivity $\si\in C^1_\sharp(Y)$, then the function $w_v$ is in $L^\infty(\RR^2)$.
On the other hand, if $w_v$ belongs to $L^\infty(\RR^2)$, then condition \eqref{L2X3'} holds.
\end{Pro}
\begin{Rem}\label{rem.curlj=0}
The criterion for the isotropic realizability of $j$ in the torus given in Proposition~\ref{pro.jz=0}, that is the boundedness of $w_v$, implies condition \eqref{L2X3'}. The converse is not clear. The reason is that the field $j$ defined by \eqref{jz=0} does not satisfy condition~\eqref{ortbasj}. More precisely, the curl of $j$
\beq\label{curljjz=0}
\curl j(x,y,z)=-\,\al'(z)\,\nabla v(x,y)+\al(z)\,\De v(x,y)\,e_z,\quad\mbox{for }(x,y,z)\in\RR^3,
\eeq
does vanish in $\RR^3$. Indeed, due to the periodicity of $\al$ and $\nabla v$ we have
\[
\int_0^1\al'(z)\,dz=\int_{Y'}\De v(x,y)\,dx dy=0.
\]
This combined with the continuity of $\al'$ and $\De v$ implies the existence of a point $(x,y,z)\in\RR^3$ such that $\al'(z)=\De v(x,y)=0$.
Therefore, we cannot use Theorem~\ref{thm.glorea}.
\end{Rem}
\noindent
{\bf Proof of Proposition~\ref{pro.jz=0}.}
Assume that the function $w_v$ of \eqref{wv} is in $L^\infty(\RR^2)$.
Then, by Theorem~\ref{thm.elerea} the gradient field $\nabla v$ is isotropically realizable in the torus, namely there exists a periodic positive conductivity $\si$, with $\si,\si^{-1}\in L^\infty_\sharp(Y')$, such that $\div(\si\nabla v)=0$ in~$\RR^2$.
Hence, there exists a stream function $u$, with $\nabla u\in L^2_\sharp(Y')^2$, such that $\si\nabla v=-\nabla^\perp u$ in $\RR^2$.
Therefore, by \eqref{jz=0} $j=\al\,\si^{-1}\nabla u$ is isotropically realizable in the torus.
\par
Conversely, assume that $j$ is isotropically realizable in the torus with a positive conductivity $\si\in C^1_\sharp(Y)$.
Then, since $j_z=0$, we have $j=\si\nabla u$ in $\RR^3$, with $\nabla u\in C^1_\sharp(Y')^2$.
Equating this equality with \eqref{jz=0} at $z=0$, we get that
\[
{\al(0)\over\si(x,y,0)}\,\nabla v(x,y)=-\nabla^\perp u(x,y),\quad\mbox{for }(x,y)\in\RR^2.
\]
Thus, $\nabla v$ is isotropically realizable in the torus with the conductivity $\al(0)\,\si^{-1}(\cdot,0)\in C^1_\sharp(Y')$, which by virtue of  Theorem~\ref{thm.elerea} implies that $w_v$ belongs to $L^\infty(\RR^2)$.
\par
Now, assume that $w_v$ is in $L^\infty(\RR^2)$. By \eqref{Xi}, \eqref{jz=0} and \eqref{curljjz=0} we have
\[
{\partial X_3\over\partial t}={j\times\curl j\over |j|^2}(X_3)={\De v(X_3)\over|\nabla v(X_3)|^2}\,\nabla v(X_3)+{\al'(X_3\cdot e_z)\over\al(X_3\cdot e_z)}\,e_z.
\]
Moreover, by \eqref{wv} we have $\div(e^{w_v}\nabla v)=0$ in $\RR^2$, or equivalently $\De v=-\nabla w_w\cdot\nabla v$ in $\RR^2$, hence
\beq\label{Xjv}
{\partial X_3\over\partial t}=-\Pi_{\nabla v}(\nabla w_v)(X_3)+{\al'(X_3\cdot e_z)\over\al(X_3\cdot e_z)}\,e_z,
\eeq
where $\Pi_{\nabla v}$ denotes  the orthogonal projection on $\RR\nabla v$.
However, since $\nabla w_v-\Pi_{\nabla v}(\nabla w_v)$ is parallel to $j$ by \eqref{jz=0} and thus orthogonal to $\partial_t X_3$ by \eqref{Xi}, we get that
\[
\ba{ll}
\dis \left|{\partial X_3\over\partial t}\right|^2 & 
\dis =-\Pi_{\nabla v}(\nabla w_v)(X_3)\cdot {\partial X_3\over\partial t}+{\al'(X_3\cdot e_z)\over\al(X_3\cdot e_z)}\,{\partial X_3\over\partial t}\cdot e_z
\\ \ecart
& \dis =-\nabla w_v(X_3)\cdot {\partial X_3\over\partial t}+{\al'(X_3\cdot e_z)\over\al(X_3\cdot e_z)}\,{\partial X_3\over\partial t}\cdot e_z
={\partial\over\partial t}\left[-\,w(X_3)+\ln\big(\al(X_3\cdot e_z)\big)\right]
\ea
\]
It follows that for any $(t,x,y,z)\in\RR^4$,
\beq\label{lenXjz=0}
\int_0^t\left|{\partial X_3\over\partial s}(s,x,y,z)\right|^2\,ds=w_v(x,y,z)-w_v\big(X_3(t,x,y,z)\big)
+\ln\left[{\al\big(X_3(t,x,y,z)\cdot e_z\big)\over\al(z)}\right].
\eeq
The function $\ln\al$ is periodic and continuous in $\RR$, hence it is bounded. Therefore, equality \eqref{lenXjz=0} shows that the boundedness of $w_v$ implies condition \eqref{L2X3'}.
\begin{Rem}\label{rem.jsh2}
Note that the field current $j$ defined by \eqref{jsh} has a zero $z$-coordinate, and
\beq
j=\nabla^\perp v\quad\mbox{with}\quad v(x,y):=\cosh x-y,
\eeq
so that we could {\em a priori} use the above method. However, in this case the solution of the gradient system \eqref{Z} is given by
\beq
Z(t,x,y)=\begin{pmatrix} \dis 2\,{\rm argtanh}\left(e^t\tanh\,(x/2)\right) \\ -\,t+y \end{pmatrix},
\eeq
which is clearly not a global solution. Therefore, the present two-dimensional approach does not work for the very simple current field \eqref{jsh}.
\end{Rem}
\subsection{A particular class of current fields}
Let $f,g,h\in C^2(\RR)$ be three $1$-periodic functions such that $f$ has only isolated roots in $\RR$ which are not roots of $f'$, while $g,h$ do not vanish in $\RR$.
Then, the following isotropic realizability result holds:
\begin{Pro}\label{pro.fgh}
The $Y$-periodic field $j$ defined by
\beq\label{jfgh}
j(x,y,z):=\begin{pmatrix} g(y)\,h(z) \\ f(x)\,h(z) \\ f(x)\,h(z) \end{pmatrix},\quad\mbox{for }(x,y,z)\in\RR^3,
\eeq
satisfies the conditions \eqref{divj=0} and \eqref{j.curlj=0}.
On the other hand, consider the following assertions:
\begin{itemize}
\item[$(i)$] the function $f$ does not vanish in $\RR^3$,
\item[$(ii)$] the field $j$ is isotropically realizable in the torus with a positive conductivity $\si\in C^1_\sharp(Y)$,
\item[$(iii)$] condition \eqref{L2X3'} holds.
\end{itemize}
Then, $(i)$ and $(ii)$ are equivalent conditions and they imply assertion $(iii)$. Moreover, when $g=h=1$ all three assertions $(i)$, $(ii)$ and $(iii)$ are equivalent.
\end{Pro}
\begin{proof}
Condition \eqref{divj=0} clearly holds. We have
\beq\label{curljfgh}
\curl j=\begin{pmatrix} f(x)\,\big(g'(y)-h'(z)\big) \\ g(y)\,\big(h'(z)-f'(x)\big) \\ h(z)\,\big(f'(x)-g'(y)\big) \end{pmatrix},\quad\mbox{for }(x,y,z)\in\RR^3,
\eeq
hence condition \eqref{j.curlj=0} is also satisfied. However, note that condition \eqref{diffX32} does not hold.
\par\ms\noindent
$(i)\Rightarrow(ii)$. If $f$ does not vanish in $\RR$, we can write
\beq\label{ufgh}
j(x,y,z)=f(x)\,g(y)\,h(z)\,\nabla u(x,y,z),\;\;\mbox{with}\;\;u(x,y,z):=\int_0^x{dt\over f(t)}+\int_0^y{dt\over g(t)}+\int_0^z{dt\over h(t)}.
\eeq
Therefore, $j$ is isotropically realizable in the torus with the conductivity
\[
\si(x,y,z):=\big|f(x)\,g(y)\,h(z)\big|>0,
\]
which belongs to $C^1_\sharp(Y)$.
\par\ms\noindent
$(ii)\Rightarrow(i)$. More precisely, we will prove that if $f$ vanishes in $\RR$, then $j$ is not isotropically realizable with any positive continuous conductivity $\si(x,y,z)$ which is $Y'$-periodic with respect to $(y,z)$.
\par
To this end, assume by contradiction that both $f$ vanishes in $\RR$ and $j$ is isotropically realizable with a positive continuous conductivity $\si(x,y,z)$ which is $Y'$-periodic with respect to $(y,z)$, $Y'=[0,1]^2$. Let $a<b$ be such that $f(a)=0$ and $f(b)\neq 0$. Starting from the equality $\curl(\si^{-1} j)=0$, integrating by parts over the cube $[a,b]\times Y'$, and denoting by $n$ the outside normal, we have
\beq\label{int=0}
\ba{ll}
0 & \dis =\int_{[a,b]\times Y'}\curl(\si^{-1} j)\cdot e_y\,dx\,dy\,dz=\int_{\partial([a,b]\times Y')}(n\times \si^{-1}j)\cdot e_y\,ds
\\ \ecart
& \dis =\int_{\{a,b\}\times Y'}(n\times\si^{-1}j)\cdot e_y\,ds+\int_{[a,b]\times\partial Y'}(n\times \si^{-1}j)\cdot e_y\,ds.
\ea
\eeq
The integral over $[a,b]\times\partial Y'$ in \eqref{int=0} is equal to zero due to the $Y'$-periodicity of $\si^{-1}j$ with respect to $(y,z)$. Hence, using that $f(a)=0$ and $(e_x\times j)\cdot e_y=-j_z=-f(x)\,g(y)$, it follows that the integral over $\{a,b\}\times Y'$ in \eqref{int=0} satisfies
\beq
0=-\int_{Y'}f(b)\,g(y)\,\si^{-1}(b,y,z)\,dy\,dz.
\eeq
This leads to a contradiction, since the function $(y,z)\mapsto f(b)\,g(y)\,\si^{-1}(b,y,z)$ is continuous and has a constant sign in $Y'$.
\par\ms\noindent
$(ii)\Rightarrow(iii)$. This is a straightforward consequence of Corollary~\ref{cor.perrea} $i)$.
\par\ms\noindent
{\it $(iii)\Rightarrow(i)$, when $g=h=1$.} Assume by contradiction that $f$ vanishes at some point $a\in\RR$. Since $f'(a)\neq 0$, we may assume that, for instance, there exists a real number $b>a$ such that $f>0$ in $(a,b]$ and $f'>0$ in $[a,b]$.
\par
By \eqref{jfgh} and \eqref{curljfgh} we have
\[
j\times \curl j=\begin{pmatrix} 2f(x)f'(x) \\ -f'(x) \\ -f'(x) \end{pmatrix},\quad\mbox{for }x\in\RR,
\]
hence the flow $X_3$ of \eqref{Xi} reads as
\beq\label{Xf}
\left\{\ba{ll}
\dis x'(t)={2f(x(t))f'(x(t))\over 2f^2(x(t))+1}, & x(0)=x
\\ \ecart
\dis y'(t)=-\,{f'(x(t))\over 2f^2(x(t))+1}, & y(0)=y
\\ \ecart
\dis z'(t)=-\,{f'(x(t))\over 2f^2(x(t))+1}, & z(0)=z,
\ea\right.
\quad\mbox{for }(x,y,z)\in\RR^3.
\eeq
Define the function $F$ in $(a,b]$ by
\[
F(x):=\int_b^x{2f^2(s)+1\over 2f(s)f'(s)}\,ds,\quad\mbox{for }x\in(a,b].
\]
The function $F$ is an increasing bijection from $(a,b]$ onto $(-\infty,0]$, and the solution of the first equation of \eqref{Xf} is given by
\[
x(t)=F^{-1}\big(t+F(x)\big),\quad\mbox{for }t\leq 0,
\]
where $F^{-1}$ denotes the reciprocal of $F$.
Making the change of variables $r=x(s)$, we have for any $t\leq 0$,
\[
\int_0^t \big(y'(s)\big)^2\,ds=\int_x^{x(t)}\left({f'(r)\over 2f^2(r)+1}\right)^2{2f^2(r)+1\over 2f(r)f'(r)}\,dr
=\int_x^{x(t)}{f'(r)\over 2f(r)\,\big(2f^2(r)+1\big)}\,dr.
\]
Then, since $x(t)$ tends to $a$ as $t\to-\infty$ and $\dis f(r)\mathop{\sim}_{a}f'(a)\,(r-a)$, we get that
\beq
\int_{-\infty}^0\left|{\partial X_3\over\partial s}(s,x,y,z)\right|^2\,ds\geq-\lim_{t\to-\infty}\int_0^t \big(y'(s)\big)^2\,ds=-\int_x^{a}{f'(r)\over 2f(r)\,\big(2f^2(r)+1\big)}\,dr=\infty.
\eeq
Therefore, the $L^2(\RR)$-norm of $t\mapsto\partial_t X_3(t,x,y,z)$ is infinite for any $x\in(a,b]$.
This proves the implication $(iii)\Rightarrow (i)$, when $g=h=1$.

\end{proof}
\begin{Exa}\label{rem.f}
Consider the particular case of \eqref{jfgh} where $g=h=1$. When the function $f$ vanishes in $\RR$, the field $j\in C^2_\sharp(Y)^3$ still satisfies conditions \eqref{divj=0} and \eqref{j.curlj=0}. Since condition \eqref{ortbasj} does not hold, Theorem~\ref{thm.glorea} does not apply. However, the field $j$ is actually isotropically realizable in the whole space $\RR^3$, but not in the torus due to Proposition~\ref{pro.fgh}.
\par
It is not obvious how to derive an explicit conductivity associated with $j$, but we now proceed to do so.
To this end, we may assume that $f\in C^2(\RR)$ is a $1$-periodic function satisfying $f(0)=0$, $f'(0)\neq 0$, and $f>0$ in $(0,1)$. Then, $j$ lies in the plane $\{j_y=j_z\}$, so that we can apply the procedure of Section~\ref{ss.jplan} based on the representation of two-dimensional divergence-free functions by orthogonal gradients. This combined with the approach of \cite{BMT} (Proposition~2.11) allows us to construct a conductivity $\si\in C^1(\RR^3)$ such that $\curl(\si^{-1}j)=0$ in $\RR^3$, as follows:
\par
Let $F$ be the function defined in $(0,1)$ by
\beq\label{Ff}
F(x):=\int_{{1\over 2}}^x{ds\over f(s)},\quad\mbox{for }x\in (0,1).
\eeq
The function $F$ is a $C^1$-diffeomorphism from $(0,1)$ onto $\RR$. Then, denoting by $F^{-1}$ its reciprocal, an admissible conductivity $\si$ is given by
\beq\label{sifF}
\si(x,y,z):=\left\{\ba{cl}
\dis {f(x)\over f\big(F^{-1}(y+z+F(x-n))\big)} & \mbox{if }x\in(n,n+1)
\\ \ecart
e^{-f'(0)(y+z)} & \mbox{if }x=n,
\ea\right.
\quad\mbox{for }n\in\ZZ,
\eeq
which is $1$-periodic with respect to $x$.
Let us prove that $\si\in C^1(\RR^3)$, and $\curl(\si^{-1}j)=0$ in $\RR^3$.
\par
For $x\in(0,1)$ and $(y,z)\in\RR^2$, set $t:=F^{-1}(y+z+F(x))\in(0,1)$. Since $f'(0)\neq 0$ and $f\in C^2(\RR)$, we have for $n=0,1$,
\beq
\ba{ll}
\dis \left|\,y+z-{1\over f'(0)}\,\ln\left({t-n\over x-n}\right)\right| & \dis =\left|\,F(t)-F(x)-{1\over f'(0)}\,\ln\left({t-n\over x-n}\right)\right|
\\ \ecart
& \dis =\left|\,\int_x^t\left({1\over f(s)}-{1\over f'(0)(s-n)}\right)ds\,\right|\leq C\,|t-x|,
\ea
\eeq
which implies that
\beq\label{limtx}
\lim_{x\to n}\,{f(t)\over f(x)}=\lim_{x\to n}\,{t-n\over x-n}=e^{f'(0)(y+z)},\quad\mbox{for }n=0,1.
\eeq
This shows the continuity of the function $\si$. Moreover, we have for any $x\in(0,1)$,
\beq\label{Dsif}
{\partial\si\over\partial x}={f'(x)-f'(t)\over f(t)}\quad\mbox{and}\quad{\partial\si\over\partial y}={\partial\si\over\partial z}=-\,{f(x)\,f'(t)\over f(t)},
\eeq
which together with \eqref{limtx} implies that $\nabla\si$ has finite limits as $x\to 0$ or $1$.
Therefore, the function $\si$ belongs to $C^1(\RR^3)$.
\par
Set $w:=\ln\si$. By \eqref{sifF} and \eqref{Dsif} we have for any $x\in(0,1)$,
\[
{\partial w\over\partial x}={1\over\si}\,{\partial\si\over\partial x}={f'(x)-f'(t)\over f(x)}\quad\mbox{and}\quad
{\partial w\over\partial y}={\partial w\over\partial z}={1\over\si}\,{\partial\si\over\partial y}=-f'(t),
\]
which implies that $w$ solves the equations
\beq\label{eqwf}
f(x)\,{\partial w\over\partial x}-f'(x)={\partial w\over\partial y}={\partial w\over\partial z}\quad\mbox{in }\RR^3.
\eeq
By \eqref{jfgh} and \eqref{curljfgh} with $g=h\equiv 1$, equations \eqref{eqwf} lead to the equation
\[
\curl j-\nabla w\times j=0\quad\mbox{in }\RR^3,
\]
or equivalently, $\curl(e^{-w}j)=0$ in $\RR^3$.
\par
Therefore, the current field $j=\big(1,f(x),f(x)\big)$ is isotropically realizable with the conductivity $\si\in C^1(\RR^3)$ defined by \eqref{Ff} and \eqref{sifF}. Note that the function $\si$ is $1$-periodic with respect to $x$, but is not periodic with respect to $(y,z)$ in accordance with Proposition~\ref{pro.fgh}.
Finally, the isotropic realizability of $j$ in $\RR^3$ can be written
\[
\begin{pmatrix} 1 \\ f(x) \\ f(x) \end{pmatrix}=
\sigma(x,y,z)\,\nabla\left(x+\int_0^{y+z}f\big(F^{-1}(s+F(x-n))\big)\,ds\right),\quad\mbox{for }x\in(n,n+1),\ n\in\ZZ,
\]
where the conductivity $\si$ is defined by \eqref{sifF}.
\end{Exa}
\par\bs\noindent
{\bf Acknowledgments.} GWM thanks the {\em National Science Foundation} for support through grant DMS-1211359. Also the authors thank Andrejs Treibergs for his comments.

\begin{thebibliography}{10}

\bibitem{Arn}{\sc V.I.~Arnold}: {\em Ordinary differential equations}, translated from the third Russian edition by R.~Cooke, Springer Textbook, Springer-Verlag, Berlin 1992, pp.~334.

\bibitem{Boy}{\sc J.B. Boyling}: ``Carath\'eodory's principle and the global existence of an integrating factor", {\em Commun. Math. Phys.}, {\bf 10} (1968), 52-68.

\bibitem{BMT}{\sc M.~Briane, G.W.~Milton \& A.~Treibergs}: ``Which electric fields are realizable in conducting materials?", {\em ESAIM: Math. Model. Numer. Anal.}, {\bf 48} (2) (2014), 307-323.

\bibitem{Car}{\sc H.~Cartan}: {\em Calcul Diff\'erentiel}, (French) Hermann, Paris 1967, 178 pp.

\bibitem{Mil}{\sc G.W.~Milton}: {\em The Theory of Composites}, Cambridge Monographs on Applied and Computational Mathematics, Cambridge University Press, Cambridge 2002, pp.~719.

\bibitem{PSS1}{\sc J.B.~Pendry, D. Schurig \& D.R.~Smith}: ``Controlling Electromagnetic Fields", {\em Science}
, {\bf 312} (5781) (2006), 1780-1782.

\bibitem{PSS2}{\sc J.B.~Pendry, D.~Schurig \& D.R.~Smith}: ``Calculation of material properties and ray tracing in transformation media", {\em Optic Express}, {\bf 14} (21) (2006), 9794-9804.

\end{thebibliography}
\end{document}